\documentclass[12pt,a4paper,reqno]{amsart}
\usepackage[latin1]{inputenc}
\usepackage{amsmath,amsfonts,amssymb,amsthm,bookmark,appendix,graphicx}
\usepackage{mathrsfs} 
\usepackage{dsfont} 
\usepackage{tikz}

\newcommand\copyrighttext{%
  \footnotesize 
  This is an author-created, un-copyedited version of an article accepted for publication in Mathematische Nachrichten.
  The publisher is not responsible for any errors or omissions in this version of the manuscript or any version derived from it.}
\newcommand\copyrightnotice{%
\begin{tikzpicture}[remember picture,overlay]
\node[anchor=north,yshift=-70pt] at (current page.north) {\fbox{\parbox{\dimexpr\textwidth-\fboxsep-\fboxrule\relax}{\copyrighttext}}};
\end{tikzpicture}%
}

\theoremstyle{plain}
\newtheorem{theorem}{Theorem}
\newtheorem{lemma}[theorem]{Lemma}

\theoremstyle{remark}

\newtheorem*{remark}{Remark}

\def\leqslant{\le}
\def\bq{\begin{eqnarray}}
\def\eq{\end{eqnarray}}
\def\bqq{\begin{eqnarray*}}
\def\eqq{\end{eqnarray*}}
\def\nn{\nonumber}
\def\minus {\backslash}
\def\eps{\varepsilon}

\def\fall{~\text{for all}~}

\def\d {\partial}
\def\E {\mathscr{E}}
\def\EE {\mathscr{E}(t,x(t))}
\def\F {\mathscr{F}}

\def\R {\mathbb{R}}

\def\real {\mathds{R}}
\def\N {\mathds{N}}
\def\diss {\mathscr{D}iss}

\def\nn {\nonumber}
\def\sign {\operatorname{sign}}
\def\N0 {\mathop N\limits^{\circ}}

\def\fall{~\text{for all}~}
\def\diss {\mathscr{D}iss}

\title[Regularity of weak solutions in one-dimension]{Regularity of weak solutions to rate-independent systems in one-dimension}
\author[Minh N.~Mach]{}
\subjclass{49N60.}
 \keywords{regularity, weak solutions, energetic solutions, BV solutions, rate-independent systems, SBV, piecewise $C^1$, finite jumps.}

 \email{mach@mail.dm.unipi.it}

\thanks{The author is partially supported by the PRIN 2008 grant ``Optimal mass transportation, Geometric and Functional Inequalities and Applications" and the FP7-REGPOT-2009-1 project ``Archimedes Center for Modeling, Analysis and Computation".}

\begin{document}
\copyrightnotice
\date{{November 25, 2012}}
\maketitle

\centerline{\scshape Mach Nguyet Minh}
\medskip
{\footnotesize
 \centerline{Dipartimento di Matematica}
   \centerline{Universit\`a di Pisa}
   \centerline{ Largo  Bruno Pontecorvo 5, 56127 Pisa, Italy}
} 

\begin{abstract} We show that under some appropriate assumptions, every {\it weak solution} (e.g.  energetic solution) to a given rate-independent system is of class SBV, or has finite jumps, or is even piecewise $C^1$. Our assumption is essentially imposed on the energy functional, but not convexity is required. 

\end{abstract}

\section{Introduction}
The rate independency is the property indicating to those systems which are subjected by an external loading on a time scale that is much slower than any internal time scale, but still much faster than the time so that the system reaches equilibrium, so that the inertia and kinetic energies can be neglected. The main feature of rate-independent systems is that the changes of the rate of the solutions essentially depends on the changes of the velocity of the loading, namely if the loading acts twice faster, then the solutions also respond twice faster. Rate-independent systems are used to characterized many physical phenomena involved in plasticity, phase transformation (electromagnetism, superconductivity or dry friction on surfaces), and some certain hysteresis models (shape-memory alloys, quasistatic delamination, fracture, etc.). For a detailed discussion on the rate-independent systems, we refer to the books \cite{KraPok-89, Monteiro-93, Visintin-94, BroSpr-96}.

In this paper, we are interested in the regularity of weak solutions to one-dimensional rate-independent systems. In one-dimension, a rate-independent system is characterized by an energy functional $\E \in C^1([0,T]\times \R; \R)$ and a dissipation function, which we will take the usual distance $|\cdot|$ for simplicity. A BV function $x: [0,T] \to \real$ is called a {\em weak solution} to the rate-independent system with the initial position $x_0\in \R$ if $x(0)=x_0$, and $x(\cdot)$ satisfies
\begin{itemize}
\item[(i)] the {\em weak local stability}, that 
\begin{equation} \label{eq:w-LS}
|\d_x \E(t,x(t))| \le 1  
\end{equation}
for every $t\in [0,T]$ such that $x(\cdot)$ is continuous at $t$, and
\item[(ii)] the {\em energy-dissipation upper bound}, that
\begin{equation} \label{eq:ED-upper}
\E(t_2,x(t_2)) - \E(t_1,x(t_1)) \le \int_{t_1}^{t_2} \d_t \E(s,x(s))\,ds - \diss(x(\cdot);[t_1,t_2]), 
\end{equation}
for all $0 \le t_1 \le t_2 \le T$. 
\end{itemize}
Here we define the dissipation energy
\[\diss(x(\cdot);[t_1,t_2]):= \sup \left\{ \sum_{n =1}^{N} |x(s_n)-x(s_{n-1})| \; | \; N \in \mathbb{N}, t_1 \le s_0 < s_1 < \dots < s_N \le t_2 \right\}. \]

A particular case of weak solutions is the {\em energetic solutions}, which was first introduced by Mielke and Theil \cite{MieThe-99} and further studied in \cite{MieThe-04, MaiMie-05, FraMie-06, Mielke-06}. A BV function $u: [0,T] \to \real$ is called an energetic solution to the rate-independent system with the initial position $x_0\in \R$ if $x(0)=x_0$ and $x(\cdot)$  satisfies 
\begin{itemize}
\item[(i)] the \textit{global stability}, that 
\bq \label{eq:global-stability}
\E(t,u(t)) \le \E(t,x)+|x-u(t)| 
\eq
for all  $(t,x) \in [0,T]\times \real$, and
\item[(ii)] and the \textit{energy-dissipation balance}, that 
\bq \label{eq:energy-dissipation-old}
\E(t_2,u(t_2))-\E(t_1,u(t_1))=\int_{t_1}^{t_2} {\partial_t \E(s,u(s))\,ds} - \diss( u;[t_1,t_2]) .
\eq
for all $0 \le t_1 < t_2 \le T$.
\end{itemize}
However, our notion of weak solutions also contains BV solutions \cite{MieRosSav-12}, local solutions \cite{ToaZan-09}, parametrized solutions \cite{MieRosSav-10} and epsilon-stable solutions \cite{Larsen-10}.

When the energy functional is {\it convex}, the regularity was already investigated by Mielke, Rossi and Thomas \cite{MieRos-07,ThoMie-10}. They showed that if the energy functional $\E(t,\cdot)$ is $\alpha$-convex, $\d_t \E(t,\cdot)$ is Lipschitz continuous (or H\"older continuous), and
\bqq
 |\d_t\E(t,x)| \le \lambda \,\E(t,x) \quad \forall t \in [0,T]
\eqq 
for some constant $\lambda>0$, then every energetic solution is Lipschitz continuous (or H\"older continuous, respectively). Moreover, if the energy functional has the form $\E(t,x)= W(x)-\ell(t)\,x$, where $W(x)$ is the double-well potential and $\ell(t)$ is a smooth loading, Stefanelli \cite{Ste-09} proposed a variational characterization of rate-independent evolution. Later, if $\E(t,x)= W(x)-\ell(t)\,x$ for a general smooth potential $W(x)$ and a monotone loading function $\ell(t)$, Rossi and Savar\'e \cite{RiSa-12} derived a full characterization of all energetic and BV solutions to rate-independent systems in one-dimension.

However, in the general case (in particular when the energy functional is {\it non-convex}), the solutions may behave badly, as we can see in the following 

\begin{theorem}[Any increasing function is an energetic solution] \label{thm:increasing-is-energetic}
Let $u: [0,T] \to \real$ be an arbitrary increasing and left-continuous function. Then $u$ is an energetic solution of some rate-independent system with smooth energy functional.
\end{theorem}

In this paper, we shall prove that under some certain requirements (but not convexity) on the energy functional, any weak solution is of class SBV. Moreover, we give sufficient conditions ensuring that every weak solution has only finitely many jumps, and it is piecewise $C^1$-smooth in one-dimensional case. In recent years, many authors investigated one-dimensional rate-independent models for the propagation of a single crack \cite{KMZ08,ToZa-09,Neg10,Rac12,KnSc13}. However, it is not so obvious to check whether the energy functionals in these models satisfy assumptions (H1)-(H5). We hope to come back to the SBV-regularity with easier-to-check assumptions on energy functionals  in higher dimensions in future work.

\section{Main results}

Our first regularity result deal with the SBV property of weak solutions. We shall need a technical condition.
\begin{itemize}

\item[(H1)] $\E(t,x)$ is of class $C^3$ and the set
\[\{ (t,x)\in (0,T)\times \mathbb{R} \;|\; \d_x \E (t,x)\in \{-1,1\}, {\partial _{xx}}\E(t,x) = {\partial _{xxx}}\E(t,x) = 0\} \]
has only finitely many elements.
\end{itemize}

Note that no convexity is imposed. We have 

\begin{theorem}[SBV regularity]\label{thm:SBV} Assume that {\rm (H1)}  holds true. Then every BV function $x(\cdot)$ satisfying the weak local stability  (\ref{eq:w-LS}) and the energy-dissipation upper bound (\ref{eq:ED-upper}) must be of class SBV.  
\end{theorem}

\begin{remark} The SBV regularity still holds if the set in (H1) is at most countable (instead of finite). Moreover, the result can be generalized in higher dimensions as follows (see \cite{Minh-Thesis} for a detailed proof).  We assume that the energy functional $\E(t,x)$ is of class $C^3$ and the set
\[ \{ (t,x)\in (0,T)\times \mathbb{R}^d \; |\; |\nabla_x \E(t,x)|=1, G(t,x)= ({\nabla_x}\E(t,x))\cdot({\nabla_x} F(t,x))= 0  \} \]
is at most countable, where the function $F(t,x)$ is defined by 
\[
F (t,x):= ({\nabla_x}\E(t,x))\cdot H(t,x) \cdot ({\nabla_x}\E(t,x))^T
\]
with the Hessian matrix 
\[{[{H(t,x)}]_{ij}} := ({\partial _{{x_i}}}{\partial _{{x_j}}}\E)(t,x).\]
Then every BV function $x: [0,T] \to \real^d$ (with $d \ge 1$) satisfying the weak local stability  (\ref{eq:w-LS}) and the energy-dissipation upper bound (\ref{eq:ED-upper}) must be of class SBV.  
\end{remark}

In the next result, we consider the differentiability of weak solutions. By a technical reason, we have to replace the above weak local stability by the strong-local stability:
\bq \label{eq:s-LS}
z=x(t) ~{\rm is~ a ~local~ minimizer~of~ the~ functional}~ z\mapsto \E(t,z)+|z-x(t)|
\eq 
for every $t\in [0,T]\minus J$, where $J$ is the jump set of $x(\cdot)$, which will be assumed to be finite. Notice that, because of this condition, Theorem \ref{thm:differentiable} is only valid for a more restrictive class of weak solutions (i.e. energetic solutions). Moreover, we shall replace the condition (H1) on the energy functional by some of the following.

\begin{itemize}

\item[(H2)] The set
\bqq
  \left\{ {(t,x)\in (0,T)\times \mathbb{R} \;|\; \d_x \E (t,x)\in \{-1,1\}, \partial _{xx}\E(t,x) = \partial _{xt} \E(t,x) = 0,} \right.\hfill\\
  \left. {[\d_{xxt}\E (t,x) ]^2= \d_{xtt}\E (t,x)\cdot \d_{xxx} \E (t,x)} \right\} 
\eqq
has only finitely many elements.

\item[(H3)] The set
\[\{ (t,x)\in (0,T)\times \mathbb{R} \;|\; {\d_x \E (t,x)\in \{-1,1\}, \partial _{xx}}\E(t,x) = {\partial _{xt}}\E(t,x) = 0\} \]
has only finitely many elements.

\item[(H4)] The set
\[\{ (t,x)\in [0,T]\times \mathbb{R}\; |\; {\d_x \E (t,x)\in \{-1,1\}, \partial _{xt}}\E(t,x) = {\partial _{xtt}}\E(t,x) = 0\}\]
is empty.
\end{itemize}

Now we define the right and left derivatives $x'_+(t)$, $x'_-(t)$ as follows
\[ x'_+(t):=\lim_{s\downarrow t} \frac{x(s)-x(t)}{s-t},~~x'_-(t):=\lim_{s\uparrow t} \frac{x(s)-x(t)}{s-t}.
\]
and say that $s$ is an isolated point  of $I$ if there exists $\eps>0$ such that 
\[(s-\eps,s+\eps)\cap I=\{s\}.\]
We have the following theorem.

\begin{theorem}[Differentiability]\label{thm:differentiable} Assume that the BV function $x:[0,T]\to \R$ has only finitely many jump points and satisfies the strong-local stability (\ref{eq:s-LS}) and the energy-dissipation upper bound (\ref{eq:ED-upper}). Then we have the following statements.

\begin{itemize}
\item[(i)] If {\rm (H1)} holds true, then we can decompose $[0,T]$ into four disjoint sets $I_1, I_2, I_3$ and $J$ such that 
\begin{itemize}

\item For every $t\in I_1$, $x'(t)$ does not exist and either $x'_-(t) = 0$ or $x'_+(t)=0$.

\item For every $t\in I_2$, $x_-'(t)$ and $x'_+(t)$ do exist, but they are different. Moreover, $x(\cdot)$ is differentiable in a neighborhood of $t$ (except the point $t$ itself) and  
\[x'_+(t)=\lim_{s\downarrow t} x'(s),~~x'_-(t)=\lim_{s \uparrow t} x'(s).\]

\item For every $t\in I_3$, $x(\cdot)$ is differentiable at $t$, namely $x'(t)$ exists.

\item $J$ is the jump set of $x(\cdot)$.
\end{itemize}
Notice that both $I_1$ and $I_2$ are discrete sets. Moreover, if {\rm(H1)} and {\rm(H2)} holds true, then $I_1\cup I_2$ is also a discrete set.

\item[(ii)] If {\rm (H1)} and {\rm (H3)} hold true, then there exists a set $I$ of isolated points such that for any $t\in (0,T)\minus I$, the (classical) derivative $x'(t)$ exists. Moreover, the function $x'(\cdot)$ is continuous on  $(0,T)\minus I$. 

\item[(iii)] If {\rm (H1)}, {\rm (H3)} and {\rm (H4)} hold true, then there exist {\em finite} disjoint open intervals $\{I_n\}_{n\ge 1}^M$ such that $[0,T]=\cup_{n\ge 1}^M \overline{I_n}$, and $x(\cdot)$ is $C^1$ on any interval $I_n$.
\end{itemize}
\end{theorem}

In Theorem \ref{thm:differentiable}, we have required, as a-priori, that the solution has finitely many jump points. In the last result, we give a sufficient condition on the energy functional to remove this assumption. 
\begin{itemize}
\item[(H5)] The set
\[\{ (t,x)\in [0,T]\times \mathbb{R} \;|\; {\d_x \E (t,x)\in \{-1,1\}, \partial _{xx}}\E(t,x) = 0\}\]
is empty.

\end{itemize}

\begin{theorem}[Finite jumps]\label{thm:finite-jumps} Assume that {\rm (H5)} holds true. Then every BV function $x:[0,T]\to \R$ satisfying the weak local stability (\ref{eq:w-LS}) and the energy-dissipation upper bound (\ref{eq:ED-upper}) must have only finitely many jumps.
\end{theorem}

The proofs of the previous theorems are provided in the next sections.

\section{Proof of Theorem \ref{thm:increasing-is-energetic}}

We start by the following lemma

\begin{lemma}\label{le:classical-2} If $u:[0,T]\mapsto \real$ is an increasing function, then there exists a smooth function $g: [0,T] \times \real \to \real$ such that
\begin{itemize}
\item[] $g(t,x) \in [-1,0) \fall t \in [0,T] $ and $\fall x < u(t)$,
\item[] $g(t,x) 
\in (0,1] \fall t \in [0,T] $ and $\fall x > u(t)$,
\item[] $g(t,x) = 0 \fall t \in [0,T] $ and $\fall x = u(t)$,
\item[] $g(t,x) = 1$ for $x \ge M$,
\item[] $g(t,x) = -1$ for $x \le -M$,
\end{itemize}
where $M$ is such that $1-M \le u(t) \le M-1$ for every $t$.
\end{lemma}

The proof of  Lemma \ref{le:classical-2} can be found in the Appendix. Now we give

\begin{proof}[Proof of Theorem \ref{thm:increasing-is-energetic}] Fixing an $x_0 \in \mathbb{R}$ and taking $g$ from Lemma \ref{le:classical-2}. We choose the energy functional $\E(t,x)$ as follows
\[ \E(t,x):= \int_{x_0}^x g(t,y)+1 \; dy.\] 
Then $\E$ is smooth and satisfies $\d_x \E(t,u(t))=-1$, $\d_x \E(t,x) \in (-1,0]$ if $x> u(t)$, and $\d_x \E(t,x) \in [-2,-1)$ if $x<u(t)$ for all $t$. Moreover, it is easy to check that $|\partial_t \E(t,x)| \le $ const. By adding a constant into $\E$ if needed, we can assume that $|\partial_t \E(t,x)| \le \lambda \E(t,x)$ for every $(t,x)$. 

We shall prove that $u$ is an energetic solution of the system $(\E,|\cdot|,u(0))$. It is known that $x(\cdot)$ is an energetic solution to the system $(\E,|\cdot|,u(0))$ if the following three conditions hold (see Proposition 5.13 \cite{AlbDeS-11}, or a simplified version in Proposition 1.4 \cite{Minh-Thesis}). 
\begin{itemize}
\item[(i)] $x(\cdot)$ is left-continuous.
\item[(ii)] $\diss(x(\cdot);[0,T]) = |x(T)-x_0|$.
\item[(iii)] For all $t \in [0,T]$, $x(t)$ minimizes the functional $x \mapsto \E(t,x) + |x-x_0|$ for $x \in \real$.
\end{itemize}
Thus it remains to check that $u$ satisfies the condition (iii). We shall use the fact that for all $t$, $\partial_x \E(t,u(t))=-1$, $\partial_x \E(t,x) \in (-1,0]$ if $x>u(t)$, and $\partial_x \E(t,x) \in [-2,-1)$ if $x<u(t)$. We distinguish two cases.

{\bf Case 1: $x>u(t)$.} By the smoothness of $\E$, we can write
\bqq
\E(t,x) = \E(t,u(t)) + \int_{u(t)}^x \d_x \E(t,z)dz >  \E(t,u(t)) + \int_{u(t)}^x (-1) = \E(t,u(t)) + u(t) -x.
\eqq

{\bf Case 2: $x<u(t)$.} Similarly to Case 1, we write
\bqq
\E(t,x) = \E(t,u(t)) - \int_{x}^{u(t)} \d_x \E(t,z)dz > \E(t,u(t)) - \int_x^{u(t)} (-1) = \E(t,u(t)) + u(t) -x.
\eqq

Thus in both cases, we have
\bqq
\E(t,x) + |x-u(0)| > \left[\E(t,u(t)) + u(t) - x\right] + \left[x-u(0)\right]  = \E(t,u(t)) + u(t)-u(0).
\eqq

In summary, $u(t)$ is the unique minimizer for the functional $x \mapsto \E(t,x) + |x-x(0)|$ over $x \in \real$. This completes the proof.
\end{proof}

\section{Proof of Theorem \ref{thm:SBV}}

In this section, we prove Theorem \ref{thm:SBV}.

\begin{proof} {\bf Step 1.} 
Thanks to Proposition 1.5 \cite{Minh-Thesis}, we can assume that $x(\cdot)$ is right-continuous. By dividing $(0,T)$ into smaller intervals if necessary, we can assume that the set
\[\{ (t,x)\in (0,T)\times \mathbb{R}\; |\; \d_x \E(t,x) \in \{-1,1\},{\partial _{xx}}\E(t,x) = {\partial _{xxx}}\E(t,x) = 0\}\]
is empty.

\text{}\\
{\bf Step 2.} Since $x(\cdot)$ is a BV function in $1$-dim which is right-continuous, there is a real-valued Radon measure $\mu$ such that 
\[ x(t) = {\rm const} + \mu((0,t])~\fall t\in [0,T] . \]
By Lebesgue Decomposition Theorem we can write
\[ \mu= f dx + \mu_s\]
where $f\in L^1$ and $\mu _s = \left. \mu  \right|_S$ with
\[ S = \left\{ {t \in (0,T)\; | \; \mathop {\lim }\limits_{h \downarrow 0} \frac{{|\mu |(t - h,t + h)}}{h} = \infty } \right\}.\]

Let $J$ be the jump set of $x(\cdot)$. We split $\mu_s$ into the Cantor part
 $\mu_c:={\left. \mu  \right|_{S \backslash J}}$ and the jump part $\mu_J:={\left. \mu  \right|_J}$. To show that $x(\cdot)$ is of $SBV$, we need to prove that $\mu_c=0$. 

\text{}\\
{\bf Step 3.} Next, we shall use the following lemmas, which will be proved later.

\begin{lemma}\label{le:msA=0} For any BV function $x:[0,T] \to \real$ which is right-continuous, the set
\[A := \left\{ {t \in (0,T)\minus J \; |\; \liminf_{h \to 0 } \left| {\frac{{x(t + {h}) - x(t)}}{{{h}}}} \right| < \infty } \right\}\]
has $|\mu_s|$-measure $0$.
\end{lemma}

\begin{lemma} \label{le:msB=0} Assume that the BV function $x: [0,T] \to \real$ satisfies the weak local stability (\ref{eq:w-LS}) and the energy-dissipation upper bound (\ref{eq:ED-upper}). If (H1) holds true, then the set
\[ B := \left\{ {t \in (0,T)\backslash J \; | \; \mathop {\lim }\limits_{h \to 0} \left| {\frac{{x(t + h) - x(t)}}{h}} \right| = \infty } \right\}\]
is at most countable. Therefore,  $|\mu_s|(B)=0$.
\end{lemma}

\text{}\\
{\bf Step 4.} Since $\mu_c$ is the restriction of $\mu_s$ on $(0,T) \minus J$, $\mu_c =0$ if $|\mu_s|((0,T)\minus J)=0$. Notice that $(0,T) \minus J =A \cup B$. Hence, lemmas \ref{le:msA=0} and \ref{le:msB=0} ensure that $|\mu_{s}|((0,T)\minus J)=0$. This completes the proof of Theorem \ref{thm:SBV}.
\end{proof}

It remains to verify Lemma \ref{le:msA=0} and Lemma \ref{le:msB=0}. Lemma \ref{le:msA=0}  is a general fact of BV functions, and its proof can be found in the Appendix. On the other hand, the proof of Lemma \ref{le:msB=0} is based on the following observation, which is a key property of weak solutions to rate-independent systems.

 \begin{lemma}\label{le:-1<=dx<=1-1-dim} Assume that the BV function $x:[0,T] \to \real$ satisfies the weak local stability (\ref{eq:w-LS}) and the energy-dissipation upper bound (\ref{eq:ED-upper}). Then we have 
\[\partial_{x}\E(t,x(t))\in \{-1,1\}\]
for all $t\notin J\cup {\rm int}( N \cup J) $. Here we denote by $J$ the jump set of $x(\cdot)$ and $N:=\{t\in (0,T) \;|\; x'(t)=0\}$ is the null set of the derivative of $x(\cdot)$. 
\end{lemma}

\begin{proof} {\bf Step 1.} First, we show that if $t\notin N\cup J$, then $\partial_{x}\E(t,x(t))\in \{-1,1\}.$

Since $t\notin N$, we can find a sequence $t_n\to t$ and $t_n\ne t$ such that 
\bq \label{eq:liminf|x'|>0-1-dim}
\liminf_{n\to \infty} \left| {  \frac{{x(t_n) - x(t)}}{t_n-t} } \right|>0.
\eq

{\bf Case 1.} Assume that $t_n\downarrow t$. From the energy-dissipation upper bound, one has
\[ \E (t_n, x(t_n)) - \E(t,x(t))\le  \int\limits_t^{t_n} {{\partial _t}\E(s,x(s))ds}  - \diss (x(\cdot);[t,t_n]).\]
 Using Taylor's expansion on the left-hand side and the continuity of $s\mapsto \d_t \E(s,x(s))$ on the right-hand side, we obtain
\bqq
&~&\d_t \EE \cdot (t_n-t) + \d_{x} \EE \cdot (x(t_n)-x(t)) + o(x(t_n)-x(t))+o(t_n-t) \hfill\\
&\le & (t_n-t) \cdot {{\partial _t}\E(t,x(t))}  - \diss(x(\cdot);[t,t_n]) + o(t_n-t).
\eqq
Dividing this inequality by $|x(t_n)-x(t)|$ and using (\ref{eq:liminf|x'|>0-1-dim}), we obtain  
\bq \label{eq:dx*x'<0}
\d_x \EE \cdot \frac{x(t_n)-x(t)}{|x(t_n)-x(t)|}  \le - \frac{\diss( x(\cdot);[t,t_n])}{|x(t_n)-x(t)|}+o(1) \le -1+o(1). 
\eq
Consequently, $|\d_x \EE| \ge 1$. On the other hand, $|\d_x \EE| \le 1$ by the weak local stability (\ref{eq:w-LS}). Thus $|\d_x \EE|=1$. 

{\bf Case 2.} Assume that $t_n\uparrow t$. From the energy-dissipation upper bound, one has
\[ \E (t_n, x(t_n)) - \E(t,x(t))\ge  \int\limits_t^{t_n} {{\partial _t}\E(s,x(s))ds} +\diss (x(\cdot);[t_n,t]).\]
Following the above proof, we obtain
\bqq
\d_x \EE \cdot \frac{x(t_n)-x(t)}{|x(t_n)-x(t)|}  \ge \frac{\diss( x(\cdot);[t_n,t])}{|x(t_n)-x(t)|}+o(1) \ge 1+o(1). 
\eqq
This also implies that  $|\d_x \EE|=1$.
\text{}\\\\
{\bf Step 2.} We show that if $t\notin J$ and $t\notin {\rm int}( N \cup J)$, then $\partial_{x}\E(t,x(t))\in \{-1,1\}.$

Since $t\notin {\rm int}( N \cup J)$, there exists a sequence $t_n\to t$ such that $t_n\notin N\cup J$ for all $n\ge 1$. By the previous step, $\partial_{x}\E(t_n,x(t_n))\in \{-1,1\}$ for all $n\ge 1$. Moreover, since $x(\cdot)$ is continuous at $t$, we get
\[\partial_{x}\E(t_n,x(t_n)) \to \d_{x}\EE.\]
Therefore, $\partial_{x}\EE \in \{-1,1\}.$
\end{proof}

As an easy consequence of Lemma \ref{le:-1<=dx<=1-1-dim}, we have 

\begin{lemma}\label{le:notin-N-J-E} Assume that $x:[0,T] \to \real$ has bounded variation and satisfies the weak local stability (\ref{eq:w-LS}) and the energy-dissipation upper bound (\ref{eq:ED-upper}). If $t\notin J \cup {\rm int}(N \cup J)$ and $\d_{xx}\EE\ne 0$, then for any sequence $t_n\to t$ such that $t_n\notin J\cup {\rm int}(N \cup J)$ and $t_n\ne t$ for all $n\ge 1$, one has
\[ \lim_{n\to \infty}\frac{x(t_n)-x(t)}{t_n-t}= -\frac{\d_{xt}\EE}{\d_{xx}\EE}.\]
Here $J$ is the jump set of $x(\cdot)$, and $N:=\{t\in (0,T) \;|\; x'(t)=0\}$.
\end{lemma}

\begin{proof} By Lemma \ref{le:-1<=dx<=1-1-dim} we have
$\partial_{x}\E(t,x(t))\in \{-1,1\}$ and $\partial_{x}\E(t_n,x(t_n))\in \{-1,1\}$ for all $n\ge 1$. Due to the continuity of the function $s\mapsto \d_x\E(s,x(s))$ at $s=t$, we obtain 
\[\partial_{x}\E(t_n,x(t_n))=\partial_{x}\E(t,x(t))\]
for $n$ large enough. Therefore, by Taylor's expansion, 
\bqq
0&=&\partial_{x}\E(t_n,x(t_n))-\partial_{x}\E(t,x(t))\hfill\\
&=& \partial_{xt} \E(t,x(t)) \cdot (t_n-t) + \partial_{xx} \E(t,x(t)) \cdot (x(t_n)-x(t)) + o(t_n-t)+o(x(t_n)-x(t)),
\eqq
we get
\[\mathop {\lim }\limits_{n \to \infty } \frac{{x({t_n}) - x(t)}}{{{t_n} - t}} =  - \frac{{{\partial _{xt}}\E(t,x(t))}}{{{\partial _{xx}}\E(t,x(t))}}.
\]
Here we have assumed that $\partial_{xx}\E(t,x(t))\ne 0$.
\end{proof}

Now we are able to give 
\begin{proof}[Proof of Lemma \ref{le:msB=0}] Let $J$ be the jump set of $x(\cdot)$, $N:=\{t\in (0,T) \;|\; x'(t)=0\}$  and $E=\{t\in (0,T)\;|\; \partial_{xx}\E(t,x(t))=0\}$. By Assumption (H1) and by dividing the interval $(0,T)$ to be many smaller intervals if necessary, we have that $\partial_{xxx}\E(t,x(t))\ne 0$ for any $t\in E$. For an arbitrary point $t\in (0,T)$, we have one of the following cases.

{\bf Case 1.} If $t\in N\cup J$, then $t\notin B$, by the definition of $B$. 

{\bf Case 2.} If $t$ is an accumulation point of $(0,T)\minus (N\cup J)$ and $t\notin E$, then we can find a sequence $t_n\to t$ such that $t_n\notin N \cup J $ and $t_n\ne t$ for all $n\ge 1$. By Lemma \ref{le:notin-N-J-E},
\[ \lim_{n\to \infty}\frac{x(t_n)-x(t)}{t_n-t}= -\frac{\d_{xt}\EE}{\d_{xx}\EE}.\]
Thus in this case, $t\notin B$.

{\bf Case 3.} If $t\notin J$ and $t$ is an accumulation point of $E$, then we can find a sequence $s_n\in E$, $s_n\to t$. Using Taylor's expansion again, we get
\bqq
0&=& {\partial _{xx}}\E({s_n},x({s_n}))- {\partial _{xx}}\E(t,x(t))  \hfill \\
   &=& {\partial _{xxt}}\E(t,x(t))\cdot ({s_n} - t) + {\partial _{xxx}}\E(t,x(t))\cdot (x({s_n}) - x(t))+ o({s_n} - t) + o(x({s_n}) - x(t)). 
   \eqq
Since $\partial_{xxx}\E(t,x(t))\ne 0$, we arrive at 

   \[\mathop {\lim }\limits_{n \to \infty } \frac{{x({s_n}) - x(t)}}{{{s_n} - t}} =  - \frac{{{\partial _{xxt}}\E(t,x(t))}}{{{\partial _{xxx}}\E(t,x(t))}},\]
which is a finite number. Thus $t\notin B$. 

{\bf Conclusion.} In summary, if $t\in B$, then either $t$ is an isolated point of $(0,T)\minus (N\cup J)$, or $t$ is an isolated point of $E$. Therefore, $B$ is at most countable. Since $\mu_s(\{t\})=0$ for any $t\in B\subset (0,T)\minus J$, we have $|\mu_s|(B)=0.$ This ends the proof of Lemma \ref{le:msB=0}.
\end{proof}
The proof of Theorem \ref{thm:SBV} is completed.

\section{Proof of Theorem \ref{thm:differentiable}}

In this section we shall prove Theorem \ref{thm:differentiable}. We shall always assume that $\E$ is of class $C^3$. We shall also denote by $J$ the jump set of $x(\cdot)$, 
\[N:=\{t\in (0,T) \;|\; x'(t)=0\}\]
and \[E:=\{t\in (0,T)\;|\; \partial_{xx}\E(t,x(t))=0\}.\]

\subsection{Proof of Theorem \ref{thm:differentiable} (ii)}

To prove  Theorem \ref{thm:differentiable} (ii), besides Lemma \ref{le:-1<=dx<=1-1-dim} and Lemma \ref{le:notin-N-J-E}, we need some other preliminary results. 

\begin{lemma}\label{le:dxx>=0} Assume that the BV function $x:[0,T] \to \real$ is {\em continuous} and satisfies the strong local stability (\ref{eq:s-LS}) and the energy-dissipation upper bound (\ref{eq:ED-upper}). If $t\notin {\rm int}(N)$, then $\partial_{xx}\E(t,x(t))\ge 0.$ Moreover, if $t\in E$, then $\partial_{x}\E(t,x(t))\cdot \partial_{xxx}\E(t,x(t))\le 0.$
\end{lemma}

\begin{proof} {\bf Step 1.} By Lemma \ref{le:-1<=dx<=1-1-dim}, we have $\partial_{x}\E(t,x(t))\in \{-1,1\} \fall t \notin {\rm int}(N)$. On the other hand, from the strong local  stability (\ref{eq:s-LS}), by using Taylor's expansion for $\E(t,\cdot)$ up to the second order, we can write
\bq \label{eq:dxx>=0-a}
 \EE &\le & \E(t,x(t))+|z-x(t)|+\partial_{x} \EE \cdot (z-x(t))\nn\hfill\\
&~&+ \partial_{xx} \EE \cdot \frac{(z-x(t))^2}{2} + o(|z-x(t)|^2)
\eq
for $z$ near $x(t)$. If $\partial_{x}\E(t,x(t))=-1$, then taking a sequence $z_n \downarrow x(t)$ in (\ref{eq:dxx>=0-a}) we get $\partial_{xx} \EE \ge 0$. If $\partial_{x}\E(t,x(t))=1$, then taking a sequence $z_n \uparrow x(t)$ in (\ref{eq:dxx>=0-a}), we also get $\partial_{xx} \EE \ge 0$. 

\text{}\\
{\bf Step 2.} Now assuming $\partial_{xx}\E(t,x(t))= 0$, we shall prove that $\partial_{x}\E(t,x(t))\cdot \partial_{xxx}\E(t,x(t))\le 0$. Using the above stability and Taylor's expansion for $\E(t,\cdot)$ up to the third order, we get
\bq \label{eq:dxx>=0-b}
 \EE &\le & \E(t,x(t))+|x-x(t)|+\partial_{x} \EE \cdot (x-x(t))\nn\hfill\\
&~&+ \partial_{xxx} \EE \cdot \frac{(x-x(t))^3}{6} + o(|x-x(t)|^3).
\eq
If $\partial_{x}\E(t,x(t))=-1$, then taking a sequence $x_n \downarrow x(t)$ in (\ref{eq:dxx>=0-b}) we get $\partial_{xxx} \EE \ge 0$. If $\partial_{x}\E(t,x(t))=1$, then taking a sequence $x_n \uparrow x(t)$ in (\ref{eq:dxx>=0-a}), we get $\partial_{xxx} \EE \le 0$. Thus in both cases, we always have $\partial_{x}\E(t,x(t))\cdot \partial_{xxx}\E(t,x(t))\le 0.$
\end{proof}

\begin{lemma}\label{le:dx*x'<=0} Assume that the BV function $x:[0,T] \to \real $ satisfies the weak local stability (\ref{eq:w-LS}) and the energy-dissipation upper bound (\ref{eq:ED-upper}). Then for all $t\in (0,T)\minus J$, one has 
\[\limsup_{s\to t}  \left\{ { {\partial_{x} \EE \cdot \frac{x(s)-x(t)}{s-t}} } \right\} \le 0.\]
Here $J$ is the jump set of $x(\cdot)$.
\end{lemma}

\begin{proof} We shall show that for any sequence $t_n\to t$ and $t_n\ne t$ then
\[\limsup_{n\to \infty}  \left\{ { {\partial_{x} \EE \cdot \frac{x(t_n)-x(t)}{t_n-t}} } \right\} \le 0.\]
Of course, we may assume that 
\bqq 
\liminf_{n\to \infty} \left| {  \frac{{x(t_n) - x(t)}}{t_n-t} } \right|>0
\eqq
and either $t_n\downarrow t$ or $t_n\uparrow t$.

{\bf Case 1.} If $t_n\downarrow t$, then repeat the argument in the proof of Lemma \ref{le:-1<=dx<=1-1-dim}, we obtain again the inequality (\ref{eq:dx*x'<0}) 
\[ \partial_x \E(t,x(t))\cdot \frac{x(t_n)-x(t)}{|x(t_n)-x(t)|} \le -1 + o(1),\]
and either $\partial_x \E(t,x(t))=1$ or $\partial_x \E(t,x(t))=-1$.
This implies that
\[ \d_x \EE \cdot \lim_{n\to \infty}  \frac{x(t_n)-x(t)}{|x(t_n)-x(t)|}  =-1.\]
or
\[\d_x \EE \cdot \lim_{n\to \infty} \sign \left( {  \frac{x(t_n)-x(t)}{t_n-t}  } \right)=-1.\]

{\bf Case 2.} If $t_n\uparrow t$, then similarly, one has
\[\d_x \EE \cdot  \lim_{n\to \infty} \frac{x(t_n)-x(t)}{|x(t_n)-x(t)|}  =1,\]
and hence
\[\d_x \EE \cdot  \lim_{n\to \infty} \sign \left( {  \frac{x(t_n)-x(t)}{t_n-t}  } \right)=-1.\]

Thus in all cases, we have
\[\d_x \EE \cdot  \lim_{n\to \infty} \sign \left( {  \frac{x(t_n)-x(t)}{t_n-t}  } \right)=-1.\]
and the conclusion follows. 
\end{proof}

\begin{lemma}\label{le:dxt=0} Assume that the BV function $x:[0,T] \to \real$ is {\em continuous} and satisfies the strong local stability (\ref{eq:s-LS}) and the energy-dissipation upper bound (\ref{eq:ED-upper}). Let $t\in {\rm int}[(0,T)\minus {\rm int}(N) ]$ such that $\partial_{xx}\E(t,x(t))=0$ and $\d_{xxx}\E(t,x(t)) \ne 0$. Then $\partial_{xt}\EE=0$. 
\end{lemma}

\begin{proof} 

{\bf Step 1.} Take an arbitrary sequence $t_n\to t$, $t_n\ne t$, $t_n \in {\rm int}[(0,T)\minus {\rm int}(N)]$. By Lemma \ref{le:-1<=dx<=1-1-dim} and the continuity of the function $s\mapsto \partial_x \E(s,x(s))$, we have 
\[ \d_x \E(t_n,x(t_n))=\d_x \EE \in \{-1,1\}\]
for all $n$ large enough. Using Taylor's expansion and the assumption $\d_{xx}\EE=0$, we have
\bq \label{eq:E-dxt=0-a}
0&=&\d_x \E(t_n,x(t_n))-\d_x \EE\nn\hfill\\
&=&\d_{xt}\EE \cdot (t_n-t)+o(x(t_n)-x(t)))+o(t_n-t).
\eq
Thus we can conclude that $\partial_{xt}\EE=0$ if we can find a sequence $t_n\to t$ such that
\[ \limsup_{n\to \infty} \frac{|x(t_n)-x(t)|}{|t_n-t|}<\infty.\]

\text{}\\
{\bf Step 2.} For an arbitrary sequence $s_n\to t$, $s_n\ne t$, $s_n \in {\rm int}[(0,T)\minus {\rm int}(N)]$, by Lemma \ref{le:dxx>=0} we have 
\[\d_{xx} \E(s_n,x(s_n))\ge 0=\d_{xx}\E(t,x(t)).\]
Therefore, using Taylor's expansion we obtain
\bqq 
0&\le & \d_{xx} \E(s_n,x(s_n))-\d_{xx}\E(t,x(t)) \nn\hfill\\
&=& \d_{xxt}\EE \cdot (s_n-t)+ \d_{xxx} \EE \cdot (x(s_n)-x(t))+ o(x(s_n)-x(t))+o(s_n-t).
\eqq
Choosing $s_n\uparrow t$ and dividing the above inequality for $(s_n-t)<0$, we have
\bq \label{eq:E-dxt=0-b}
0\ge \d_{xxt}\EE + (\d_{xxx} \EE +o(1))\cdot \frac{x(s_n)-x(t)}{s_n-t}+o(1).
\eq

\text{}\\
{\bf Step 3.} Since $\partial_x \EE \in \{-1,1\}$, we distinguish two cases.

{\bf Case 1.} Assume $\partial_x \EE =-1$. Then $\d_{xxx}\EE >0$ by Lemma \ref{le:dxx>=0}. Therefore, (\ref{eq:E-dxt=0-b}) implies that
\[ \limsup_{n\to \infty} \frac{x(s_n)-x(t)}{s_n-t}\le -\frac{\d_{xxt}\EE}{\d_{xxx}\EE}<\infty.\]
On the other hand, by Lemma \ref{le:dx*x'<=0},
\[\liminf_{n\to \infty}  \left\{ { {\frac{x(s_n)-x(t)}{s_n-t}} } \right\} \ge 0.\]
Therefore, 
\[\limsup_{n\to \infty} \frac{|x(s_n)-x(t)|}{|s_n-t|}< \infty.\]

{\bf Case 2.} Assume $\partial_x \EE =1$. Similarly, we have $\d_{xxx}\EE <0$ by Lemma \ref{le:dxx>=0}, and hence
\[ \liminf_{n\to \infty} \frac{x(s_n)-x(t)}{s_n-t}\ge -\frac{\d_{xxt}\EE}{\d_{xxx}\EE}>-\infty.\]
Moreover, by Lemma \ref{le:dx*x'<=0},
\[ \limsup_{n\to \infty}  \left\{ { {\frac{x(s_n)-x(t)}{s_n-t}} } \right\} \le 0.\]
Thus 
\[ \limsup_{n\to \infty}  \frac{|x(s_n)-x(t)|}{|s_n-t|}< \infty.\]

\text{}\\
{\bf Step 4.} In summary, if $s_n\uparrow t$, then we always have
\[\limsup_{n\to \infty} \frac{|x(s_n)-x(t)|}{|s_n-t|}< \infty.\]
Therefore, choosing $t_n=s_n$ in (\ref{eq:E-dxt=0-a}), we conclude that $\d_{xt}\EE=0.$
\end{proof}

\begin{lemma}\label{le:dN0-dxt=0} Assume that the BV function $x:[0,T]\to \R$ is {\em continuous} and satisfies the strong local stability (\ref{eq:s-LS}) and the energy-dissipation upper bound (\ref{eq:ED-upper}). If $t\in (0,T)$ is an accumulation point of $\partial \N0 $, then $\d_{xt}\EE=0.$ Moreover, if $\partial_{xx}\E(t,x(t))\ne 0$, then $x'(t)=0$ and $\d_{xtt}\EE=0$.

Here for convenience, we denote by  $\d \N0 $ the boundary of ${\rm int}(N)$.
\end{lemma}

\begin{proof} {\bf Step 1.} Since $t$ is an accumulation point of $\d \N0 $, we can find $a_n\to t$, $b_n\to t$ such that $(a_n,b_n)\subset {\rm int}(N) $ and $a_n,b_n\in \d \N0 $. By Lemma \ref{le:-1<=dx<=1-1-dim}, and the continuity of $s\mapsto \d_{x}\E(s, x(s))$, one has, for $n$ large enough,
\[\d_{x}\E(a_{n}, x(a_{n}))=\d_x \EE=\d_{x}\E(b_{n}, x(b_{n}))\in \{-1,1\}.\]
 
Note that for all $s\in [a_{n},b_{n}]$, $x(s)=c_n$, a constant independent of $s$. Consider the one-variable function 
\[s\mapsto f_n(s):=\d_x \E(s,c_n).\]
Since $f_n(a_{n})=f_n(b_{n})$, by Rolle's Theorem, we can find a number $s_n\in (a_{n},b_{n})$ such that $f_n'(s_n)=0$. This means $\d_{xt}\E(s_n,x(s_n))=0$. Since $s_n\to t$, one has 
\[0=\d_{xt}\E(s_n,x(s_n))\to \d_{xt}\E(t,x(t)).\]
\text{}\\
{\bf Step 2.} Now we assume that $t\notin E$. We distinguish two cases.

{\bf Case 1}. Let $t_n\notin {\rm int}(N) $, $t_n\ne t$ and $t_n\to t$. Then by Lemma \ref{le:notin-N-J-E} we have 
\[ \lim_{n\to \infty} \frac{x(t_n)-x(t)}{t_n-t}= -\frac{\d_{xt}\EE}{\d_{xx}\EE}=0.\]

{\bf Case 2}. Let $s_n\in {\rm int}(N) $ and $s_n\to t$. Since $t$ is an accumulation point of $\d \N0 $,  we can assume that $s_n\in (a_n,b_n)\subset {\rm int}(N) $ with $a_n,b_n\in \d \N0 $. Using Case 1, one has
\[ \lim_{n\to \infty}\frac{x(a_n)-x(t)}{a_n-t}=\lim_{n\to \infty}\frac{x(b_n)-x(t)}{b_n-t}=0.\]
On the other hand, since $x'(s)=0$ when $s \in (a_n,b_n)$, we have $x(s_n)=x(a_n)=x(b_n)$. Therefore,
\[\left| {\frac{{x({s_n}) - x(t)}}{{{s_n} - t}}} \right| \leqslant \max \left\{ {\left| {\frac{{x({a_n}) - x(t)}}{{{a_n} - t}}} \right|,\left| {\frac{{x({b_n}) - x(t)}}{{{b_n} - t}}} \right|} \right\} \to 0\]
as $n\to \infty$.

Thus in summary, for any sequence $t_n\to t$ and $t_n\ne t$ we always have
\[ \lim_{n\to \infty} \frac{x(t_n)-x(t)}{t_n-t}\to -\frac{\d_{xt}\EE}{\d_{xx}\EE}=0.\]
This means $x'(t)=0$.
\\\text{}\\
{\bf Step 3.} Now we show that if we assume furthermore that $t\notin E$, then $\d_{xtt}\EE=0.$ 

Since $t\notin E$ and the function $s\mapsto \d_{xx}\E(s,x(s))$ is continuous at $s=t$, we have $s\notin E$ if $s$ is in a neighborhood of $t$ (recall that $E$ is closed). In particular, if $s$ is in a neighborhood of $t$ and $s \in {\rm int}[(0,T)\minus{\rm int}(N) ]$, then $\d_{xx}\E(s,x(s))>0$ by Lemma \ref{le:dxx>=0}. Moreover, if $s \notin J$, then 
\[x'(s)=-\frac{\d_{xt}\E(s,x(s))}{\d_{xx}\E(s,x(s))}\]
by Lemma \ref{le:notin-N-J-E}. Using Lemma \ref{le:dx*x'<=0}, we conclude that 
\bq \label{eq:dxt*x'>=0}
\d_{xt}\E(s,x(s))\cdot \d_x \E(s,x(s)) \ge 0.
\eq

Let us assume that $\d_x \EE=-1$ (the other case, $\d_x \EE=1$, can be treated by the same way). If $s$ is in a neighborhood of $t$, $s \notin J$ and $s \in {\rm int}[(0,T)\minus {\rm int}(N) ]$, then $\d_x \E(s,x(s))<0$, and hence $\d_{xt}\E(s,x(s))\le 0$ by (\ref{eq:dxt*x'>=0}). Using the continuity of $s\mapsto \d_{xt}\E(s,x(s))$, we have
\[\d_{xt}\E(a_n,x(a_n))\le 0~\text{and}~\d_{xt}\E(b_n,x(b_n))\le 0\]
for $n$ large enough, where $\{a_n\}, \{b_n\}$ are taken as in Step 1.

On the other hand, it was already shown in Step 1 that there exists $t_n\in (a_n,b_n)$ such that $\d_{xt}\E(t_n,x(t_n))=0.$ Therefore, the function $g(s):=\d_{xt}\E(s,x(s))$ has a local maximizer $s_n\in (a_n,b_n)$.  Therefore, for $n$ large enough,
\[\d_{xtt}\E(s_n,x(s_n))= g'(s_n)=0.\]
Since $s_n\to t$, by taking the limit as $n\to \infty$ we obtain $\d_{xtt}\E(t,x(t))= 0$.
\end{proof}

Now we are able to give 
\begin{proof}[Proof of Theorem \ref{thm:differentiable} (ii)] Since $x(\cdot)$ has finite jump points and (H1), (H3) hold true, by dividing $(0,T)$ into the subintervals if necessary, we may further assume that $x(\cdot)$ is continuous on $(0,T)$ and 
\[\{ (t,x)\in (0,T)\times \mathbb{R} \;|\;\d_x \E (t,x)\in \{-1,1\}, {\partial _{xx}}\E(t,x) = {\partial _{xxx}}\E(t,x) = 0\}=\emptyset,\]
\[\{ (t,x)\in (0,T)\times \mathbb{R} \;|\;{\d_x \E (t,x)\in \{-1,1\}, \partial _{xx}}\E(t,x) = {\partial _{xt}}\E(t,x) = 0\}=\emptyset.\]

We denote by $I_1$ the set of isolated points of $\d \N0 $. It remains to consider when $t\notin I_1$. We distinguish the following cases.

{\bf Case 1.} If $t\in {\rm int}(N) $, then $x'(t)=0$. Moreover, if $s$ is in a neighborhood of $t$ then $x'(s)=0$. Therefore, $x'(\cdot)$ is continuous at $t$. 

{\bf Case 2.} If $t\in {\rm int}[(0,T)\minus {\rm int}(N) ]$, then by Lemma \ref{le:dxt=0} we have $t\notin E$. Therefore, by Lemma \ref{le:notin-N-J-E}, 
\[ x'(t)=-\frac{\d_{xt}\E (t, x(t))}{\d_{xx}\E (t,x(t))}.\]
Since the same formula also holds true for any $s$ in a neighborhood of $t$, we have that $x'(\cdot)$ is continuous at $t$. 

{\bf Case 3.} If $t$ is an accumulation point of $\d \N0 $, then $\d_{xt}\EE=0$ by Lemma \ref{le:dN0-dxt=0}. Therefore, $t\notin E$. By Lemma \ref{le:dN0-dxt=0} one has $x'(t)=0$. Next, we shall show that if $t_n\to t$ and $t_n\notin I_1$, then $x'(t_n)\to x'(t)=0$. Indeed, if $t_n\in {\rm int}[(0,T)\minus {\rm int}(N) ]$, then 
\[ x'(t_n)=-\frac{\d_{xt}\E (t_n, x(t_n))}{\d_{xx}\E (t_n,x(t_n))}\to -\frac{\d_{xt}\E (t, x(t))}{\d_{xx}\E (t,x(t))}=0.\]
Otherwise, if $t_n\in {\rm int}(N) $ or $t_n$ is an accumulation point of $\d\N0 $, then we already have $x'(t_n)=0$.

In summary, if $t\in (0,T)\minus I_1$ one has
\[x'(t) = \left\{ \begin{gathered}
   - \frac{{{\partial _{xt}\EE}}}{{{\partial _{xx}\EE}}} ~~{\rm if}~t\in {\rm int}[(0,T)\minus {\rm int}(N) ] , \hfill \\
  0~~~~~~~~~~~~~~~~~~~~{\rm otherwise,}\hfill \\ 
\end{gathered}  \right.\]
and $x'(\cdot)$ is continuous on $(0,T)\minus I_1.$ This completes the proof of Theorem \ref{thm:differentiable} (ii).
\end{proof}

\begin{remark} In general (when the jump set of $x(\cdot)$ and the sets in (H1) and (H3) are finite, instead of empty), the set $I$ in the statement of Theorem \ref{thm:differentiable} (ii) contains the following points: the isolated points of $\d \N0 $ (namely the set $I_1$ in the above proof), the jump points, and the points $t$ such that $\d_x \E (t,x(t))\in \{-1,1\}$, $\partial _{xx}\E(t,x(t))=0$ and either $\partial _{xxx}\E(t,x(t)) = 0$ or $\partial _{xt}\E(t,x(t)) = 0$.
\end{remark}

\subsection{Proof of Theorem \ref{thm:differentiable} (iii)}

To prove Theorem \ref{thm:differentiable} (iii), we need the further preliminary lemmas.

\begin{lemma}\label{le:tn-notin-N0-t-notin-E} Assume that the BV function $x: [0,T] \to \real$ is {\em continuous} and satisfies the strong local stability (\ref{eq:s-LS}) and the energy-dissipation upper bound (\ref{eq:ED-upper}). Assume furthermore that 
\[t\notin {\rm int}(N) ,\d_{xx}\EE=\d_{xt}\EE=0~~{and}~\d_{xxx}\E(t,x(t)) \ne 0.\]
Then the limits 
\[\mathop {\lim }\limits_{s \notin {\rm int}(N)  ,s \downarrow t} \frac{{x(s) - x(t)}}{{s - t}}~~ ~~and~~ \mathop {\lim }\limits_{s \notin {\rm int}(N) ,s \uparrow t} \frac{{x(s) - x(t)}}{{s - t}}~,  \]
exist and they are two solutions to the equation (w.r.t. $X$)
\bq \label{eq:tam-thuc}
\d_{xtt}\EE + 2\d_{xxt}\EE \cdot X + \d_{xxx}\EE \cdot X^2 =0.
\eq

Moreover, if there is a sequence $t_n\to t$ such that 
\[t_n\ne t,t_n\notin {\rm int}(N)~~{and}~\d_{xx}\E(t_n,x(t_n))=0\]
for all $n\ge 1$, then the equation (\ref{eq:tam-thuc}) has a unique solution $X=-\d_{xxt}\EE/\d_{xxx}\EE$. 

Here recall that $N:=\{t\in (0,T) \;|\; x'(t)=0\}$.
\end{lemma} 

\begin{proof} {\bf Step 1.} Let $t_n\to t$ and $t_n\notin {\rm int}(N) $. We have $\d_{x}\E(t_n,x(t_n))=\d_{x}\E(t,x(t))$ by Lemma \ref{le:-1<=dx<=1-1-dim} and the continuity of $s\mapsto \d_x \E(s,x(s))$ at $s=t$. Using Taylor's expansion we obtain
\bqq
0&=&\d_{x}\E(t_n,x(t_n))-\d_{x}\E(t,x(t))\hfill\\
&=& \d_{xtt}\EE\cdot (t_n-t)^2+ 2\d_{xxt}\EE \cdot (x(t_n)-x(t))\cdot(t_n-t) \hfill\\
&~&+ \d_{xxx}\EE \cdot (x(t_n)-x(t))^2 + o(|x(t_n)-x(t)|^2)+o(|t_n-t|^2).
\eqq
Dividing this equality by $(t_n-t)^2$ and taking the limit as $n\to \infty$ we get
\bq \label{eq:convergence-2-roots}
\d_{xtt}\EE + 2\d_{xxt}\EE \cdot \frac{x(t_n)-x(t)}{t_n-t} + \d_{xxx}\EE \cdot \left( {\frac{x(t_n)-x(t)}{t_n-t}} \right)^2 \to 0.
\eq
Notice that, (\ref{eq:convergence-2-roots}) also shows that the solutions of (\ref{eq:tam-thuc}) are real. Moreover, if we denote by  $X_1$ and $X_2$ the two solutions of the equation (\ref{eq:tam-thuc}) , then \[\min \left\{ {\left| {\frac{{x({t_n}) - x(t)}}{{{t_n} - t}} - {X_1}} \right|,\left| {\frac{{x({t_n}) - x(t)}}{{{t_n} - t}} - {X_2}} \right|} \right\} \to 0\]
as $n\to \infty$.

\text{}\\
{\bf Step 2.} Using Lemma \ref{le:dxx>=0} and  Taylor's expansion one has
\bq \label{eq:123-xyz}
0&\le & \d_{xx}\E (t_n,x(t_n))-\d_{xx}\EE \nn\hfill\\
&=& \d_{xxt}\EE \cdot (t_n-t) + \d_{xxx} \EE \cdot (x(t_n)-x(t))\nn\hfill\\
&~&+ o(t_n-t)+o(x(t_n)-x(t)).
\eq

By Lemma \ref{le:-1<=dx<=1-1-dim}, $\d_x \EE\in \{-1,1\}$. We distinguish two cases.

{\bf Case 1.} $\d_x \EE=-1$. In this case, by Lemma \ref{le:dxx>=0} we have $\d_{xxx}\EE >0$. Therefore, from the inequality (\ref{eq:123-xyz}), if $t_n \downarrow t$, then 
\[ \liminf_{n\to \infty}\frac{x(t_n)-x(t)}{t_n-t}\ge -\frac{\d_{xxt}\EE}{\d_{xxx}\EE};\]
while if $t_n \uparrow t$, then 
\[\limsup_{n\to \infty}\frac{x(t_n)-x(t)}{t_n-t}\le -\frac{\d_{xxt}\EE}{\d_{xxx}\EE}.\]

Since 
\[ \max\{X_1,X_2\} \ge -\frac{\d_{xxt}\EE}{\d_{xxx}\EE} \ge \min\{X_1,X_2\},\]
the convergence in (\ref{eq:convergence-2-roots}) reduces to 
\[\lim_{n\to \infty}\frac{x(t_n)-x(t)}{t_n-t}=  \max\{X_1,X_2\}~~{\rm if}~t_n\downarrow t,\]
and
\[\lim_{n\to \infty}\frac{x(t_n)-x(t)}{t_n-t}=  \min \{X_1,X_2\}~~{\rm if}~t_n\uparrow t.\]

{\bf Case 2.} If $\d_x \EE=1$, then similarly,
\[\lim_{n\to \infty}\frac{x(t_n)-x(t)}{t_n-t}=  \min\{X_1,X_2\}~~{\rm if}~t_n\downarrow t,\]
and
\[\lim_{n\to \infty}\frac{x(t_n)-x(t)}{t_n-t}=  \max \{X_1,X_2\}~~{\rm if}~t_n\uparrow t.\]

In both cases, the first conclusion of Lemma \ref{le:tn-notin-N0-t-notin-E} follows.

\text{}\\
{\bf Step 3.} Now assume that there is a sequence $t_n\to t$ such that $t_n\ne t$, $t_n\notin {\rm int}(N) $ and  $\d_{xx}\E(t_n,x(t_n))=0$ for all $n\ge 1$. Using Taylor's expansion,
\bqq 
0&= & \d_{xx}\E (t_n,x(t_n))-\d_{xx}\EE \nn\hfill\\
&=& \d_{xxt}\EE \cdot (t_n-t) + \d_{xxx} \EE \cdot (x(t_n)-x(t))+ o(t_n-t)+o(x(t_n)-x(t)),\nn
\eqq
we find that
\[ \lim_{n\to \infty}\frac{x(t_n)-x(t)}{t_n-t}=-\frac{\d_{xxt}\EE}{\d_{xxx}\EE}.\]
Thus $-{\d_{xxt}\EE}/{\d_{xxx}\EE}$ is a solution to (\ref{eq:tam-thuc}). Substituting this solution into (\ref{eq:tam-thuc}) we find that
\[ [\d_{xxt}\EE]^2 = \d_{xtt}\EE \cdot \d_{xxx}\EE ,\]
which implies that (\ref{eq:tam-thuc}) has a unique solution.  
\end{proof}

\begin{lemma}\label{le:accum-N0-dxtt=0} Assume that the BV function $x: [0,T] \to \real$ is {\em continuous} and satisfies the strong local stability (\ref{eq:s-LS}) and the energy-dissipation upper bound (\ref{eq:ED-upper}). Let $t$ be an accumulation point of $\d \N0 $, and assume either $t \notin E$, or $t \in E$ and $\d_{xxx}\E(t,x(t)) \ne 0$. Then $x'(t)=0$, and $\d_{xtt}\EE=0.$

Here recall that $N:=\{t\in (0,T) \;|\; x'(t)=0\}$ and $E:=\{t\in (0,T)\;|\; \partial_{xx}\E(t,x(t))=0\}.$
\end{lemma}

\begin{proof} Since $t$ is an accumulation point of $\d \N0 $, Lemma \ref{le:dN0-dxt=0} ensures that $\d_{xt}\EE=0$. If $t\notin E$, then Lemma \ref{le:dN0-dxt=0} also implies that $x'(t)=0$  and $\d_{xtt}\EE=0$. Therefore, it remains to consider the case $t\in E$. 
\\\text{}\\
{\bf Step 1}. Since $t$ is an accumulation point of $\d \N0 $, there exists a sequence $\{(a_n,b_n)\}$ such that $(a_n,b_n)\subset {\rm int}(N) $, $a_n,b_n\in \d \N0 $ for all $n\ge 1$, and $a_n,b_n \downarrow t$ (or $a_n,b_n \uparrow t$). By Lemma \ref{le:tn-notin-N0-t-notin-E}, we have
\[\lim_{n\to \infty}\frac{{x({a_n}) - x(t)}}{{{a_n} - t}}= \lim_{n\to \infty} \frac{{x({b_n}) - x(t)}}{{{b_n} - t}} =X_1,\]
where $X_1$ is a solution to (\ref{eq:tam-thuc}). Note that $x(s)=c_n$, a constant, when $s\in [a_n,b_n]$. Therefore, if $t_n\in [a_n,b_n]$ for all $n\ge 1$, then using the fact that $x(\cdot)$ is a constant in $[a_n,b_n]$, one has
\[\left| {\frac{{x({t_n}) - x(t)}}{{{t_n} - t}}-X_1} \right| \leqslant \max \left\{ {\left| {\frac{{x({b_n}) - x(t)}}{{{b_n} - t}}-X_1} \right|,\left| {\frac{{x({a_n}) - x(t)}}{{{a_n} - t}}-X_1} \right|} \right\} \to 0.\]
Thus if $t_n\in [a_n,b_n]$, then
\[\lim_{n\to \infty}\frac{{x({t_n}) - x(t)}}{{{t_n} - t}}=X_1.\]

\text{}\\
{\bf Step 2.} On the other hand, by Lemma \ref{le:-1<=dx<=1-1-dim} and the continuity of $s\mapsto \d_x \E (s,x(s))$ at $s = t$, we have 
\bq \label{eq:dxan=dxbn=dxt=1-1}
\d_x \E (a_n,x(a_n))=\d_x \E (t,x(t))= \d_x \E (b_n,x(b_n))\in \{-1,1\}
\eq
for $n$ large enough. Consider the one-variable function
\bq \label{eq:def-f=dx}
f_n(s):= \d_x \E(s,c_n)~~{\rm on}~s\in [a_n,b_n],
\eq
where recall that $x(s)=c_n$ for all $s \in [a_n,b_n]$. Since $f_n(a_n)=f_n(b_n)$, by applying Rolle's Theorem, we can find $t_n\in (a_n,b_n)$ such that 
\[\d_{xt}\E(t_n,x(t_n))=f_n'(t_n)=0.\]

Using Taylor's expansion we have
\bqq
 0&=&\d_{xt} \E (t_n,x(t_n))-\d_{xt}\EE \hfill\\
 &=& \d_{xtt}\EE\cdot (t_n-t) + \d_{xxt}\EE \cdot (x(t_n)-x(t))+o(t_n-t)+o(x(t_n)-x(t)).
 \eqq
Dividing this equality by $t_n-t$ and taking the limit as $n\to \infty$ we obtain
\bq \label{eq:dxtt+dxxt.x'=0}
  \d_{xtt}\EE + \d_{xxt}\EE \cdot X_1=0.
  \eq

\text{}\\
{\bf Step 3.} We show that $\d_{xtt}\EE=0$. Assume by contradiction that $\d_{xtt}\EE\ne 0$. Then from (\ref{eq:dxtt+dxxt.x'=0}), we must have $ \d_{xxt}\EE\ne 0$ and 
\[ X_1=-\frac{\d_{xtt}\EE}{\d_{xxt}\EE}\ne 0.\]
Since $X_1$ is a solution to (\ref{eq:tam-thuc}), we obtain
\bq \label{eq:delta-tam-thuc=0}
 [\d_{xxt}\EE]^2 = \d_{xtt}\EE\cdot \d_{xxx}\EE,
 \eq
which in particular implies that $X_1$ is the unique solution to (\ref{eq:tam-thuc}).

From (\ref{eq:dxan=dxbn=dxt=1-1}), let us assume that 
\bqq 
\d_x \E (a_n,x(a_n))=\d_x \E (t,x(t))= \d_x \E (b_n,x(b_n))=-1
\eqq
for $n$ large enough (the other case can be treated by the same way).

By Lemma \ref{le:dxx>=0} one has $\d_{xxx}\EE>0$. From (\ref{eq:delta-tam-thuc=0}) one has $\d_{xtt}\EE >0$. By the continuity of $s\mapsto \d_{xtt}\E (s,x(s))$ at $s=t$, we have $\d_{xtt}\E (s,x(s))>0$ when $s$ is in a neighborhood of $t$. In particular, the function $f_n(s)$ defined by (\ref{eq:def-f=dx}) satisfies 
\[f_n''(s)=\d_{xtt}\E (s,x(s))>0~~\fall s\in (a_n,b_n)\]
for $n$ large enough. 

Thus $f_n$ is strictly convex on $[a_n,b_n]$. Consequently, if we choose $s:=(a_n+b_n)/2$, then 
\[
\d_x \E(s,x(s))= {f_n}(s)< \frac{f(a_n)+f(b_n)}{2}=-1.
\]
However, this contradicts to the fact that $\d_x\E(s,x(s))\ge -1$ for all $s\notin {\rm int}(N)$ by Lemma \ref{le:-1<=dx<=1-1-dim}. Thus we must have $\d_{xtt}\EE=0$. 

\text{}\\
{\bf Step 4.} Now we show that $X_1=0$. In fact, if $\d_{xxt}\EE\ne 0$, then from $\d_{xtt}\EE=0$ and (\ref{eq:dxtt+dxxt.x'=0}) we must have $X_1=0$. Otherwise, if $\d_{xxt}\EE=0$, then $0$ is the unique solution to the equation (\ref{eq:tam-thuc}), and hence we also have $X_1=0$.

\text{}\\
{\bf Step 5.} Now we show that $x'(t)=0$. We distinguish three cases.

{\bf Case 1.} Assume that there exists $a<t$ such that $(a,t)\subset {\rm int}(N) $. It is obvious that $x'_-(t)=0=\lim\limits_{s\uparrow t} x'(s).$ It remains to show that $x'_+(t)=0$, namely to show that 
\[\lim_{n\to\infty}\frac{x(t_n)-x(t)}{t_n-t}=0\]
provided that $t_n\downarrow t$.

First, we assume that $t_n\in {\rm int}(N) $ and $t_n\downarrow t$. Note that $(t,b)\not\subset {\rm int}(N) $ for all $b>t$ (otherwise, by the continuity we have $x(a)=x(t)=x(b)$ and $t\in (a,b)\subset {\rm int}(N) $, which is a contradiction). Therefore, as in Step 1, we can choose the sequence $\{(a_n,b_n)\}$ such that $(a_n,b_n)\subset {\rm int}(N) $, $a_n,b_n\in \d \N0 $ for all $n\ge 1$, and $a_n,b_n \downarrow t$. Therefore, it follows from Step 1 and the fact that $X_1=0$
\[\lim_{n\to\infty}\frac{x(t_n)-x(t)}{t_n-t}=\lim_{n\to\infty}\frac{x(a_n)-x(t)}{a_n-t}=0.\]

Next, assume that $t_n\notin {\rm int}(N) $ and $t_n\downarrow t$. Then by Lemma \ref{le:tn-notin-N0-t-notin-E} we have
\[\lim_{n\to\infty}\frac{x(t_n)-x(t)}{t_n-t}=\lim_{n\to\infty}\frac{x(a_n)-x(t)}{a_n-t}=0.\]

Thus for any sequence $t_n\downarrow t$ we obtain
\[\lim_{n\to\infty}\frac{x(t_n)-x(t)}{t_n-t}=0.\]
Therefore, $x'_+(t)=0$. Thus $x'(t)=0$.

{\bf Case 2.} If there exists $b>t$ such that $(t,b)\subset {\rm int}(N) $, then similarly to Case 1 we have $x'(t)=0$.

{\bf Case 3.} Finally, assume that $(a,t)\not\subset {\rm int}(N) $ for all $a<t$, and $(t,b)\not\subset {\rm int}(N) $ for all $b>t$. Then by the same proof in Case 1, using the fact that $(t,b)\not\subset {\rm int}(N) $ for all $b>t$, we have $x'_+(t)=0$. Similarly, using the fact that $(a,t)\not\subset {\rm int}(N) $ for all $a<t$, we obtain $x'_-(t)=0$. Thus $x'(t)=0$. This completes our proof.
\end{proof}

Now we are able to give

\begin{proof}[Proof of Theorem \ref{thm:differentiable} (iii)] {\bf Step 1.} 
Since $x(\cdot)$ has only finite jumps and (H1), (H3), (H4) hold true, by dividing $(0,T)$ into subintervals if necessary, we may assume that $x(\cdot)$ has no jump and 
\bqq
&& \{ (t,x)\in (0,T)\times \mathbb{R} \;|\;\d_x \E (t,x)\in \{-1,1\}, {\partial _{xx}}\E(t,x) = {\partial _{xxx}}\E(t,x) = 0\}=\emptyset, \hfill\\
&& \{ (t,x)\in (0,T)\times \mathbb{R} \;|\;{\d_x \E (t,x)\in \{-1,1\}, \partial _{xx}}\E(t,x) = {\partial _{xt}}\E(t,x) = 0\}=\emptyset, \hfill\\
&&\{ (t,x)\in [0,T]\times \mathbb{R}\; |\;{\d_x \E (t,x)\in \{-1,1\}, \partial _{xt}}\E(t,x) = {\partial _{xtt}}\E(t,x) = 0\} =\emptyset.
\eqq
\text{}\\
{\bf Step 2.} Assume that $\d \N0 $ has an accumulation point $t$. Then we have $\d_{x}\EE\in \{-1,1\}$ by Lemma \ref{le:-1<=dx<=1-1-dim}. Note that if $t=0$ or $t=T$, then Lemma \ref{le:-1<=dx<=1-1-dim} is not applicable directly to $t$, but because $t$ is an accumulation point of $\d \N0 $, we can apply Lemma \ref{le:-1<=dx<=1-1-dim} to the points in $\d \N0 \cap (0,T)$ first, and then take the limit to get the conclusion at $t$.

Next, we have $\d_{xt}\EE =0$ by Lemma \ref{le:dN0-dxt=0}, and $\d_{xtt}\EE =0$ by Lemma \ref{le:dN0-dxt=0} (when $t\notin E$) and Lemma \ref{le:accum-N0-dxtt=0} (when $t\in E$). Note that these lemmas apply even if $t=0$ or $t=T$. 

Thus 
\[\d_{x}\EE\in \{-1,1\}, \d_{xt}\EE =0, \d_{xtt}\EE =0.\]
By condition (H3), this case cannot happen. Therefore, $\d \N0 $ has no accumulation point. Thus $\d \N0 $ is finite, and hence ${\rm int}(N) \cup~ {\rm int}[(0,T)\minus {\rm int}(N) ]$ is the union of finitely many open intervals. 

\text{}\\
{\bf Step 3.} Finally, if $t\in {\rm int}(N) $, then $x'(t)=0$. On the other hand, if $t\in {\rm int}[(0,T)\minus {\rm int}(N) ]$, then by Lemma \ref{le:dxt=0} we have $t\notin E$, and hence  
\[ x'(t)=-\frac{\d_{xt}\E (t, x(t))}{\d_{xx}\E (t,x(t))}\]
by Lemma \ref{le:notin-N-J-E}. Thus we can conclude that $x(\cdot)$ is of class $C^1$ in ${\rm int}(N) \cup  ~{\rm int}[(0,T)\minus {\rm int}(N) ]$. The proof is completed. 
\end{proof}

\subsection{Proof of Theorem \ref{thm:differentiable} (i)}

Finally, to obtain Theorem \ref{thm:differentiable} (i), we need the following lemma.

\begin{lemma}\label{le:t-E} Assume that the BV function $x: [0,T] \to \real$ is {\em continuous} and satisfies the strong local stability (\ref{eq:s-LS}) and the energy-dissipation upper bound (\ref{eq:ED-upper}). If $t\in {\rm int}[(0,T)\minus {\rm int}(N) ]$, $t\in E$ and $\d_{xxx}\E(t,x(t)) \ne 0$, then the right and left derivatives
\[ x'_+(t):=\lim_{s\downarrow t} \frac{x(s)-x(t)}{s-t},~~x'_-(t):=\lim_{s\uparrow t} \frac{x(s)-x(t)}{s-t},
\]
exist and they are two solutions of the equation (\ref{eq:tam-thuc}). 

Moreover, if $t$ is an accumulation point of $E$, then 
\[x'(t)=-\frac{\d_{xxt}\EE}{\d_{xxx} \EE}\]
and it is the unique solution to the equation (\ref{eq:tam-thuc}).

On the other hand, if $t$ is an isolated point of $E$, then either 
\[x'(t)=-\frac{\d_{xxt}\EE}{\d_{xxx} \EE},\]
or 
\[x'_+(t)=\lim_{s\downarrow t} x'(s),~~x'_-(t)=\lim_{s\uparrow t} x'(s).\]
Here recall that $N:=\{t\in (0,T) \;|\; x'(t)=0\}$ and $E:=\{t\in (0,T)\;|\; \partial_{xx}\E(t,x(t))=0\}.$
\end{lemma}

\begin{proof} {\bf Step 1.} Since $t\in {\rm int}[(0,T)\minus {\rm int}(N) ]$ and $t\in E$, Lemma \ref{le:dxt=0} ensures that $\partial_{xt}\EE=0$. Therefore, by Lemma \ref{le:tn-notin-N0-t-notin-E}, we get that $x'_+(t)$, $x'_-(t)$ exist and they are two solutions of the equation (\ref{eq:tam-thuc}). 

\text{}\\
{\bf Step 2.} If $t$ is an accumulation point of $E$, then by Lemma \ref{le:tn-notin-N0-t-notin-E} again, the equation (\ref{eq:tam-thuc}) has a unique solution 
$-{\d_{xxt}\EE}/{\d_{xxx}\EE}.$ Therefore,
\[ x'(t)= -\frac{\d_{xxt}\EE}{\d_{xxx}\EE}.\]
\text{}\\
{\bf Step 3.} Now we assume that $t$ is an isolated point of $E$. If the equation (\ref{eq:tam-thuc}) has a unique solution, then it must be $-{\d_{xxt}\EE}/{\d_{xxx}\EE}$, and hence  
\[ x'(t)= -\frac{\d_{xxt}\EE}{\d_{xxx}\EE}.\]
Otherwise, if the equation (\ref{eq:tam-thuc}) has two distinct solutions, then we shall show that
\[x'_+(t)=\lim_{s\downarrow t} x'(s),~~x'_-(t)=\lim_{s\uparrow t} x'(s).\]
In fact, since $t$ is an isolated point of $E$, when $s$ is in a neighborhood of $t$ we have $s\notin E$. Therefore, using Lemma \ref{le:notin-N-J-E} and L'Hopital's rule, we have, as $s\downarrow t$, 
\bqq
  x'(s) =  - \frac{{{\partial _{xt}}\E(s,x(s))}}{{{\partial _{xx}}\E(s,x(s))}} =  - \frac{{\left( {\frac{{{\partial _{xt}}\E(s,x(s)) - {\partial _{xt}}\E(t,x(t))}}{{s - t}}} \right)}}{{\left( {\frac{{{\partial _{xx}}\E(s,x(s)) - {\partial _{xx}}\E(t,x(t))}}{{s - t}}} \right)}} \to  - \frac{{{\partial _{xtt}}\E(t,x(t)) + {\partial _{xxt}}\E(t,x(t))\,x{'_ + }(t)}}{{{\partial _{xxt}}\E(t,x(t)) + {\partial _{xxx}}\E(t,x(t))\,x{'_ + }(t)}} = x{'_ + }(t) .
\eqq
Here in the last identity we have used that $x'_+(t)$ solves the equation (\ref{eq:tam-thuc}). Note that $\partial _{xxt}\E(t,x(t)) + \partial _{xxx} \E(t,x(t))\,x'_ + (t)\ne 0$ because the equation (\ref{eq:tam-thuc}) has two distinct solutions.

Similarly, as $s\uparrow t$,
\[x'(s)= - \frac{{{\partial _{xt}}E(s,x(s))}}{{{\partial _{xx}}E(s,x(s))}} \to x'_-(t).\]
The proof is completed.
\end{proof}

Thus we can now provide

\begin{proof}[Proof of Theorem \ref{thm:differentiable} (i)] {\bf Step 1.} Assume that $x(\cdot)$ has only finitely many jumps and (H1) holds. By dividing $(0,T)$ into subintervals if necessary, we may further assume that $x(\cdot)$ has no jumps and
\[\{ (t,x)\in (0,T)\times \mathbb{R}\; |\;\d_x \E (t,x)\in \{-1,1\}, {\partial _{xx}}\E(t,x) = {\partial _{xxx}}\E(t,x) = 0\}=\emptyset.\]
Thus either $t \notin E$, or $t \in E$ and $\d_{xxx} \E(t,x(t)) \ne 0$. Choose $I_3$ and $I_1$ as follows
\bqq
&&I_3:= \{ t \in (0,T) \;|\;  x(\cdot) \text{ is differentiable at } t \} ; \hfill\\
&&I_1:= \{ t \in (0,T) \minus I_3\;|\; t \text{ is an isolated point of } \d \N0 \}.
\eqq
Now we consider the case $t$ is not an isolated point of $\d \N0 $. We have the following cases.

{\bf Case 1.} If $t\in {\rm int}(N) $, then $x'(t)=0$ by definition.

{\bf Case 2.} If $t$ is an accumulation point of $\d \N0 $, then $x'(t)=0$ by Lemma \ref{le:accum-N0-dxtt=0}.


{\bf Case 3.} If $t\in {\rm int}[(0,T)\minus {\rm int}(N) ]$ and $t\notin E$, then by Lemma \ref{le:notin-N-J-E},
\[ x'(t)=-\frac{\d_{xt}\EE}{\d_{xx}\EE}.\]

{\bf Case 4.} If $t\in {\rm int}[(0,T)\minus {\rm int}(N) ]$ and $t$ is an accumulation point of $E$, then by Lemma \ref{le:t-E}, 
\[ x'(t)=-\frac{\d_{xxt}\EE}{\d_{xxx}\EE}.\]

{\bf Case 5.} If $t\in {\rm int}[(0,T)\minus {\rm int}(N) ]$ and $t$ is an isolated point of $E$, then by Lemma \ref{le:t-E}, we have either  
\[ x'(t)=-\frac{\d_{xxt}\EE}{\d_{xxx}\EE},\]
or there exist $x'_+(t)$, $x'_-(t)$ and 
\[x_+ '(t) = \mathop {\lim }\limits_{s \downarrow t} x'(s), \; x_- '(t) = \mathop {\lim }\limits_{s \uparrow t} x'(s).\]

Thus we can choose $I_2$ as follows
\[ I_2:= (0,T) \minus (I_1 \cup I_3)= \{ t \in {\rm int}[(0,T)\minus {\rm int}(N)] \; | \; t \text{ is an isolated point of } E  \text{ and } x'_-(t)\ne x'_+(t) \}.\]

\text{}\\
{\bf Step 2.} Assume that (H2) also holds. Then by dividing $(0,T)$ into subintervals again we may assume further that 
\bq \label{eq:H2-empty}
  \left\{ {(t,x)\in (0,T)\times \mathbb{R} \; |\; \d_x \E (t,x)\in \{-1,1\}, \partial _{xx}\E(t,x) = \partial _{xt} \E(t,x) = 0,} \right. &~& \nn\hfill \\
  \left. {[\d_{xxt}\E (t,x) ]^2= \d_{xtt}\E (t,x)\cdot \d_{xxx} \E (t,x)} \right\}&=&\emptyset . 
\eq

We show that in this case the set $I:=I_1 \cup I_2$ only contains isolated points. Assume by contradiction that $t$ is an accumulation point of $I$. Thus we must have a sequence $t_n\to t\in I_1$ with $t_n\in I_2$ for all $n\ge 1$. By Lemma \ref{le:dxt=0} we have $\d_{xt}\E(t_n,x(t_n))=0$ for all $n$. Since $\d_{xx}\E (t_n,x(t_n))=\d_{xt}\E(t_n,x(t_n))=0$, taking the limit as $n\to \infty$ we get $\d_{xx}\E (t,x(t))=\d_{xt}\E(t,x(t))=0$. Therefore, by the second statement of Lemma \ref{le:tn-notin-N0-t-notin-E}, the equation (\ref{eq:tam-thuc}) has a unique solution $-\d_{xxt}\EE/\d_{xxx}\EE$. This implies that 
\[[\d_{xxt}\E (t,x) ]^2= \d_{xtt}\E (t,x)\cdot \d_{xxx} \E (t,x).\]
However, since $\d_{xx}\E (t,x(t))=\d_{xt}\E(t,x(t))=0$ and $\d_x \E(t,x(t))\in \{-1,1\}$ (by Lemma \ref{le:-1<=dx<=1-1-dim}), we obtain a contradiction to the assumption (\ref{eq:H2-empty}). The proof is completed.
\end{proof}

\section{Proof of Theorem \ref{thm:finite-jumps}}

In this section, we prove Theorem \ref{thm:finite-jumps}.
\begin{proof} {\bf Step 1.} Since $x(\cdot)$ is a BV function, we have $L:=\sup_{0\le t\le T}|x(t)| <\infty$. For any $t\in [0,T]$, define
\[\F(t):=\{x\in [-L,L]: |\partial_x \E(t,x)|=1\}.\]
We shall show that there exists $\eps>0$ independent of $t$ such that if $x,y\in \F(t)$ and $x\ne y$, then $|x-y|\ge \eps$. 

Indeed, we assume by contradiction that there exists a sequence $\{t_n\}_{n=1}^\infty \subset [0,T]$ and $x_n,y_n\in \F(t_n)$ such that $x_n< y_n$ and $|x_n-y_n|\to 0$. By compactness, after passing to subsequences if necessary, we may assume that $t_n\to t_0$, $x_n\to x_0$ and $y_n\to x_0$. Using the continuity of $\partial_x \E$, we have $|\partial_x \E(t_0,x_0)|=1$.

On the other hand, since $|\partial_x \E(t_n,x_n)|^2=1=|\partial_x \E(t_n,y_n)|^2$, by applying Rolle's Theorem for the function $z\mapsto |\partial_x \E(t_n,z)|^2$, we can find an element $z_n\in (x_n,y_n)$ such that $\partial_{xx}\E(t_n,z_n)=0$. Taking $n\to \infty$, we obtain $\partial_{xx}\E(t_0,x_0)=0$.

Thus $|\partial_x \E(t_0,x_0)|=1$ and $\partial_{xx}\E(t_0,x_0)=0$, which contradicts to the assumption (H5). Therefore, there exists $\eps>0$ independent of $t$, such that $|x-y|\ge \eps$ for all $x,y\in \F(t)$ and $x\ne y$.
\\\\
{\bf Step 2.} We assume that $x(\cdot)$ jumps at $t$, namely $ x(t^-)\ne x(t^+)$, here
\[ x(t^-):=\lim_{s\uparrow t}x(s)~~{\rm and}~ x(t^+):=\lim_{s\downarrow t}x(s).\]
We shall show that $|x(t^-)-x(t^+)|\ge \eps$. 

From the weak local stability of $x(\cdot)$, we have $|\partial_x \E(t,x(t^-))|\le 1$ and $|\partial_x \E(t,x(t^+))|\le 1$. If $|\partial_x \E(t,x(t^-))|=1=|\partial_x \E(t,x(t^-))|$, then by Step 1 we already get $|x(t^-)-x(t^+)|\ge \eps$. Hence, let us assume that 
\bq \label{eq:finite-jump-1}
\min \{ |\partial_x \E(t,x(t^-))|,|\partial_x \E(t,x(t^+))| \} < 1.
\eq

Using the energy-dissipation upper bound, we get
\bq\label{eq:finite-jump-2}
|x(t^+)-x(t^-)|\le \E(t,x(t^-))-\E(t,x(t^+))=\left| \int_{x(t^+)}^{x(t^-)} \partial_x \E(t,z)dz \right| \le \int_{I}|\partial_{x}\E(t,z)|
\eq
where $I$ is the closed interval  between $x(t^-)$ and $x(t^+)$. 

From (\ref{eq:finite-jump-1}) and (\ref{eq:finite-jump-2}), we conclude that there exists $y$ between  $x(t^-)$ and $x(t^+)$ such that $|\partial_x \E(t,y)|>1$. Since $|\partial_x \E(t,x(t^-))|\le 1< |\partial_x \E(t,y)|$, there exists $z_-$ between $x(t^-)$ and $y$ such that $|\partial_x \E(t,z_-)|=1$ (here $z_-$ may be equal to $x(t^-)$). Similarly, there exists $z_+$ between $x(t^+)$ and $y$ such that $|\partial_x \E(t,z_+)|=1$ (here $z_+$ may be equal to $x(t^+)$). Since $z_+\ne z_-$, we have $|z_+-z_-|\ge \eps$ by Step 1. Thus $|x(t^+)-x(t^-)|\ge |z_+-z_-|\ge \eps$.
\text{}\\
{\bf Step 3.} Thus by Step 2, any jump step is not less than $\eps$. Since $x(\cdot)$ is a BV function, it can only have finitely many jumps.
\end{proof}

\section{Appendix: Technical proofs}

\subsection{Proof of Lemma \ref{le:classical-2}}

We start by some elementary results.

\begin{lemma}\label{le:classical-1}
For any closed set $C$ in $\real^d$, there exists a smooth function $\varphi$ such that $\varphi: \real^d \to [0,1]$ and $\varphi^{-1}(0)=C$.
\end{lemma}

\begin{proof} Since the set $\real^d\minus C$ is open, we can find a family of open balls $\{B_n\}$ such that 
\[ \real^d \minus C = \bigcup_{n\in \mathbb{N}} B_n. \]

Moreover, a classical result tells us that, for any $n \in \mathbb{N}$, there exist $\varphi_n: \real^d \to [0,1]$ such that $\varphi_n$ is of class $C^{\infty}$ and $\varphi_n^{-1}(0)=\real^d \minus B_n$. 

Take $\varphi:=\sum_{n \in \mathbb{N}} \alpha_n \varphi_n$ with $\alpha_n >0 \fall n$. This implies $\varphi^{-1}(0)=C$.

Now for every $n \in \mathbb{N}$, we choose $\alpha_n$ such that $\|D^k\varphi_n\|_{\infty} \cdot \alpha_n \le 2^{-n} \fall k=0, 1, \dots, n$. It is easy to check that $\varphi(\real^d) \in [0,1]$ and $\varphi$ is of class $C^{\infty}$. This completes the proof of Lemma \ref{le:classical-1}.
\end{proof}

\begin{lemma}\label{le:classical-3}
For any couple of disjoint closed sets $C_0, C_1$ in $\mathbb{R}^d$, there exists a smooth function $\varphi: \mathbb{R}^d \to [0,1]$ such that $\varphi^{-1}(0)=C_0, \varphi^{-1}(1)=C_1$.
\end{lemma}
\begin{proof}
Taking $\varphi_0, \varphi_1$ as in Lemma \ref{le:classical-1} such that $\varphi_0^{-1}(0) = C_0$ and $\varphi_1^{-1}(0) = C_1$.
For every $x \in \mathbb{R}^d$, we choose 
\[\varphi(x):=\frac{\varphi_0(x)}{\varphi_0(x)+\varphi_1(x)}\]
then we can check easily that $\varphi$ satisfies all requirements of Lemma \ref{le:classical-3}. 
\end{proof}

Now we are ready to give the proof of Lemma \ref{le:classical-2}.

\begin{proof} Define
\bqq
&&C_1:=\{ (t,x)\; |\; x\ge u(t^-)\},~~  D_1:=\{(t,x)\;|\; x \le -M\},\hfill\\
&&C_2:=\{ (t,x)\; |\; x\le u(t^+)\},~~  D_2:=\{(t,x)\;|\; x \ge M\}.
\eqq

Obviously, $D_1$ and $D_2$ are closed sets in $\mathbb{R}^2$. Moreover, $C_1$ and $D_1$ are disjoint, $C_2$ and $D_2$ are disjoint.

We show that $C_1$ and $C_2$ are closed sets in $\real^2$. For example, to prove that $C_1$ is closed, we need to show that if a sequence $\{(t_n,x_n)\}_{n\ge 1}\subset C_1$  converges to $(t_0,x_0)$, then $(t_0,x_0)\in C_1$, namely $x_0 \ge u(t_0^-)$. Indeed, if $s<t_0$, then for $n$ large enough we have $t_n>s$, and hence $x_n \ge u(t_n^-)\ge u(s)$. Thus $x_0=\lim x_n \ge u(s)$ for all $s<t_0$, which implies that $x_0\ge \lim_{s\uparrow t_0} u(s)=u(t_0^-)$. Thus $C_1$ is closed. Similarly, we have $C_2$ is closed.

Applying Lemma \ref{le:classical-3}, we can choose two smooth functions $g_1: \real^2 \to [0,1]$ and $g_2: \real^2 \to [0,1]$ such that
\bqq
&& g_1^{-1}(0) = C_1 ~~{\rm and}~g_1^{-1}(1) = D_1,\hfill\\
&& g_2^{-1}(0) = C_2 ~~{\rm and}~g_2^{-1}(1) = D_2.
\eqq

We define $g(t,x):=g_2(t,x)-g_1(t,x)$ for all $(t,x) \in [0,T]\times \real$ .
It is straight-forward to see that the function $g$ has all desired properties. 
\end{proof}

\subsection{Proof of Lemma \ref{le:msA=0}}
We see that Lemma \ref{le:msA=0} is verified if we can check the following result.
\begin{lemma}\label{le:derivative-of-BV-functs}
At $|\mu_s|$-almost every point $t \in [0,T]\minus J$, the left and right derivatives of $x$ at $t$ exist and are both equal to $+\infty$ or both equal to $-\infty$. Here $x:[0,T] \to \mathbb{R}$ is any right-continuous BV function, the measure $\mu$ is the weak derivative of $x$, $\mu_s$ is the singular part of $\mu$ w.r.t. Lebesgue measure, and $J$ is the jump set of $x$.
\end{lemma}
Lemma \ref{le:derivative-of-BV-functs} is somehow well-known to experts on BV functions. However, since we could not find it in any standard reference book on the subject, we give here a short sketch of proof.
First, we need the following facts which are more or less well-known.

\text{}\\
{\bf Fact 1:} Let $t_n$ be the points in the jump set $J$. We take the union of the graph of $x$ and replace every point $(t_n,x(t_n))$ by the vertical segment $S_n$ with endpoints $(t_n,x(t_n^-))$ and $(t_n,x(t_n^+))$. We call this new set the ``complete graph" of $x$ and we denote it by $G$. Notice that $x$ is right-continuous, so $x(t_n)$ is always between $x(t_n^-)$ and $x(t_n^+)$, here by $x(t_n^-)$ and $x(t_n^+)$ we mean the left and right limit of $x$ at $t_n$.

We claim that there exists a Lipschitz injective path $\gamma: [0,L] \to G$ which parametrize $G$ and has the following property: 
\begin{itemize}
\item There exists countably many pairwise disjoint closed intervals $I_n$ contained in $[0,L]$ such that the restriction of $\gamma$ to each $I_n$ parametrize the segment $S_n$.
\item Given two points $s,s'$ with $s<s'$ which do not belong to the same interval $I_n$, then $\gamma_1(s)<\gamma_1(s')$ (here and below we write $\gamma_1$ and $\gamma_2$ for the two components of $\gamma$).
\end{itemize}

Finally, by choosing $L$ properly, we can also assume that $\gamma$ is an arc-length parametrization, which means that the derivative $\dot \gamma(s)$ is a vector of norm $1$ for all $s$ where it exists (that is, almost every $s$ in $[0,L]$).

\text{}\\
{\bf Fact 2:} Since $\gamma$ is injective and Lipschitz, at $\mathcal{H}^1$-almost every point $z$ of $G$, there exists a tangent line $L_z$ intended in the classical sense, here $\mathcal{H}^1$ is the $1$-dimensional Hausdorff measure. More precisely, $L_z$ exists for all $z=\gamma (s)$ such that $\gamma$ is differentiable at $s$, and $L_z$ is the line generated by the vector $\dot \gamma (s)$.

\text{}\\
{\bf Fact 3:} Let $p$ be the projection of $G$ on the horizontal axis, and let $\lambda$ be the positive measure on $[0,T]$ which is obtained as the push-forward according to the map $p$ of the measure $\sigma$ given by the restriction of $\mathcal{H}^1$ to the graph $G$, that is, $\lambda:=p_{\#}\sigma$. Then $|\mu| \le \lambda$ and in particular $|\mu|$ is absolutely continuous w.r.t. $\lambda$.


\text{}\\
{\bf Fact 4:} We can split $G$ in two parts: 
\begin{itemize}
\item The ``vertical part" $G_v$ consists of all points $z$ where the tangent line $L_z$ exists and is vertical.
\item The ``horizontal part" $G_h$ consists of all points $z$ where the tangent line $L_z$ exists and is not vertical.
\end{itemize}
Then, we can construct the measures $\lambda_v$ and $\lambda_h$ as before. We claim that $|\mu_s| \le \lambda_v$ (actually $|\mu_s| = \lambda_v$, but we do not need this).

\text{}\\
Now, we are back to the proof of Lemma \ref{le:derivative-of-BV-functs}.
\begin{proof}[Proof of Lemma \ref{le:derivative-of-BV-functs}]
As a consequence of Fact 4, it suffices to show that for $\lambda_v$-a.e. $t \notin J$, the derivative of $x$ at $t$ exists and is $+\infty$ or $-\infty$. Indeed, one shows that this is true at every point $t \notin J$ such that the tangent line $L_z$ exists, and $z$ belongs to $G_v$ (that is, the line $L_z $ is vertical). Here $z:=(t,x(t))$.

More precisely, take $t$ and $z$ as above, and let $s$ such that $z=\gamma (s)$. Then $L_z$ is the line generated by the vector $v:=\dot{\gamma}(s)$, and since this line is vertical, we have that either $v=(0,1)$ or $v=(0,-1)$. Then one easily shows that in the first case, the right and left derivatives of $x$ at $t$ are $+\infty$, and in the second case they are $-\infty$.

In fact, assume that we are in the first case. Taking any sequence $t_n$ that converges to $t$, and let $s_n$ be such that $\gamma(s_n)=(t_n,x(t_n))$. Then
\bqq
\frac{x(t_n)-x(t)}{t_n-t} = \left[ \frac{\gamma_2(s_n)-\gamma_2(s)}{s_n-s}\right] / \left[\frac{\gamma_1(s_n)-\gamma_1(s)}{s_n-s}\right]. 
\eqq
Notice that here we have $\dot{\gamma}(s)=(0,1)$. This implies that the first quotient at the right-hand side of the above equality tends to $+1$, while the second one tends to $0$, and more precisely to $0^+$ because of the fact that $\gamma_1$ is increasing in the sense specified above in Fact 1. Thus, the limit of the quotient in the left-hand side of the formula above must be $1/0^+ = +\infty$.
\end{proof}

\text{}\\

\vspace{\baselineskip}

 \end{document}